\documentclass{amsart}

\usepackage{amsmath}
\usepackage{amsthm}
\usepackage{amsfonts}
\usepackage{amsbsy}
\usepackage{amssymb}
\usepackage[initials]{amsrefs}

\newtheorem{theorem}{Theorem}
\newtheorem{proposition}{Proposition}
\newtheorem{corollary}{Corollary}

\newtheorem{lemma}{Lemma}

\newcommand{\R}{{\mathbb R}}
\newcommand{\Z}{{\mathbb Z}}

\newcommand{\set}[2]{ \left\{ #1 \ \left| \ #2 \right. \right\} }

\newcommand{\angg}[1]{\left<  \! \left< #1 \right>  \! \right>} 
\newcommand{\re}{{\mathrm{Re}}}

\title{Multilinear oscillatory integral operators and geometric stability}
\author{Philip T. Gressman}
\thanks{PTG has been partially supported by NSF Grant DMS-1764143.}
\author{Ellen Urheim}

\begin{document}
\dedicatory{In honor of Guido Weiss, whose generous mentorship and enthusiasm for mathematics have enriched the lives of many students.}
\begin{abstract}
In this article we prove a sharp decay estimate for certain multilinear oscillatory integral operators of a form inspired by the general framework of Christ, Li, Tao, and Thiele \cite{cltt2005}. A key purpose of this work is to determine when such estimates are stable under smooth perturbations of both the phase \textit{and} corresponding projections, which are typically only assumed to be linear. The proof is accomplished by a novel decomposition which mixes features of Gabor or windowed Fourier bases with features of wavelet or Littlewood-Paley decompositions. This decomposition very nearly diagonalizes the problem and seems likely to have useful applications to other geometrically-inspired objects in Fourier analysis.
\end{abstract}
\maketitle

\section{Formulation and Introduction}

Fix some positive integer $d \geq 2$ and two smooth, real-valued functions $\Phi$ and $\rho$ on a box $B_1 := (-b_1,b_1)^{2d} \subset \R^{2d}$. We suppose that the gradient of $\rho$ is nonvanishing and let $M$ be the zero set of $\rho$, i.e., $M := \set{ x \in B_1}{ \rho(x) = 0}$. The main object of study will be the functional of the form
\begin{equation} I_\lambda(f_1,\ldots,f_{2d}) = \int_{M} e^{i \lambda \Phi(x)} \left[ \prod_{j=1}^{2d} f_j(x_j) \right] a(x) d \sigma(x) \label{myfunc} \end{equation}
where $f_1,\ldots,f_{2d}$ are measurable functions on $\R$, $a$ is smooth with compact support in $B_0 := [-b_0,b_0]^{2d} \subset B_1$, and $d \sigma$ is Lebesgue measure on $M$.
Our goal is to prove multilinear $L^p$ inequalities for \eqref{myfunc} which sharply capture the best-possible decay of the norm as a function of $|\lambda|$ and are stable under sufficiently smooth perturbations of $\Phi$, $\rho$, and $a$. 

This question is motivated by various projects in the literature over the last 20 years which seek to better understand the general phenomena that produce norm decay for oscillatory integral operators. One of these core programs has been to understand and quantify stability of such operators, which is a key concern in a number of highly influential works, including Carbery, Christ, and Wright \cite{ccw1999}, Phong, Stein, and Sturm \cite{pss2001}, and Carbery and Wright \cite{cw2002}.
The present article is also intimately connected to the broad program of Christ, Li, Tao, and Thiele \cite{cltt2005} to classify nondegeneracy of multilinear oscillatory integrals in a fully general way; see 
\cite{noz2018,noz2018II,zengthesis,greenblatt2008II,co2014,gx2016,xiao2017,ggx2018}
for just a few of the most recent developments furthering this program.

It is a curious feature of the previous works that none are able to establish decay rates greater than $|\lambda|^{-1/2}$ for multilinear oscillatory integral operators which act on functions of a single real variable as \eqref{myfunc} does. It is of course possible to construct multilinear oscillatory integral operators for which the decay is indeed much greater than $|\lambda|^{-1/2}$, including, for example, the operator
\[ \tilde I_\lambda(f_1,\ldots,f_d) := \int_{\R^d} e^{i \lambda (\Phi_1(x_1,x_2) + \Phi_2(x_2,x_3) + \cdots + \Phi_{d-1}(x_{d-1},x_d))} a(x) \prod_{j=1}^d f_j(x_j) dx, \]
which, for suitable cutoff functions $a$ can be shown by the standard H\"{o}rmander theory and Fubini's Theorem to satisfy the inequality
\[ |\tilde I_\lambda(f_1,\ldots,f_d)| \lesssim |\lambda|^{-\frac{d-1}{2}} ||f_1||_{L^2(\R)} ||f_d||_{L^2(\R)} \prod_{j=2}^{d-1} ||f_j||_{L^\infty(\R)} \]
provided that $\partial_{x_i x_{i+1}}^2 \Phi_i$ is nonvanishing and suitably regular for each $i$. In a recent talk, Christ \cite{christtalk2019} observed that operators exhibiting such high decay are strongly unstable, meaning that arbitrarily small $C^\infty$ (and in his example, polynomial) perturbations of the phase drastically reduce the decay rate. We would like to add to this deep observation another, more elementary one: when the number of (one-variable) functions in the oscillatory integral exceeds the dimension of the space on which the integral is taken, using linear projections also leads immediately to geometrically unstable inequalities. By this we mean that an oscillatory integral operator of the form
\[ \int_{\R^d} e^{i \lambda \Phi(x)} a(x) \prod_{j=1}^{N} f_j(\pi_j(x)) dx \]
for $N > d$ and the $\pi_j$ being linear projections from $\R^d$ to $\R$ in general position will fail to exhibit any $\lambda$-decay at all if any one of the linear projections $\pi_j$ is replaced by $\pi_j + \epsilon \Phi$ for any nonzero $\epsilon$ (since then the phase $\Phi$ would then be a linear combination of the $\pi_j$'s, which destroys $\lambda$-decay, as explained in \cite{cltt2005}).  This is the particular motivation behind the object \eqref{myfunc}: the underlying geometric structure of the defining function $\rho$ naturally \textit{permits} smooth perturbations of the associated projections, and the degree of multilinearity of \eqref{myfunc} exceeds the dimension of $M$ by $1$, so in this regime, multilinear oscillatory integrals with linear projections already exhibit this strong sort of geometric instability.

Because of our particular emphasis on stability, we will pay slightly more attention to constants than is commonly done. Throughout the paper, the notation $A \lesssim B$ will be used to indicate that $A \leq C B$ for some $C$ which depends only on the following ``admissible'' constants: the dimension $d$,
 fixed finite constants $C_\rho$, $C'_\rho$, $C_\Phi$, and $C_a$ such that \begin{equation}
\max_{|\alpha| \leq 2d + 3} \sup_{x \in B_1} |\partial^\alpha \rho(x)| \leq C_{\rho}, \label{finiterhonorms}
\end{equation}
\begin{equation}
\max_{|\alpha| \leq 2d + 2}  \sup_{x \in B_1} |\partial^{\alpha} \Phi(x)| \leq C_{\Phi}, \label{finitephinorms}
\end{equation}
\begin{equation}
\max_{|\alpha| \leq 2d + 2} \sup_{x \in B_0} |\partial^{\alpha} a(x)| \leq C_{a}, \label{finiteanorms}
\end{equation}
and \begin{equation} (C_{\rho}')^{-1} \leq \min_{i=1,\ldots,2d} \inf_{x \in B_1} | \partial_{i} \rho(x)|, \label{implicit}
\end{equation}
and the finite, positive constant $c$ appearing in the hypothesis \eqref{mainhyp} of the main theorem below.
Also for purposes of stability, it will always be assumed that the parameter $\lambda$ in \eqref{myfunc} satisfies
\begin{equation} |\lambda|^{-\frac{1}{2}} \leq \min \left\{b_1 - b_0, 1 \right\} \label{lambdaisntsmall} \end{equation}
where $b_0$ and $b_1$ are the parameters associated to the boxes $B_0$ and $B_1$ in the definition \eqref{myfunc}. Our theorem is as follows:
\begin{theorem}
Suppose that there is a positive constant $c$ such that for every $x \in B_1$ and every $(\tilde \tau,\tau) \in \R^2$ such that $\tilde \tau^2 + \tau^2 = 1$, the indices $\{1,\ldots,2d\}$ may be partitioned into two sets $\{i_1,\ldots,i_d\}$ and $\{j_1,\ldots,j_d\}$ such that \label{mainthm}
\begin{equation} \left| \det \! \left[ \begin{array}{cccc} 
\partial_{{i_1}} \rho(x) \! \! & \partial^2_{{i_1} {j_1}} \! \left( \tilde \tau  \Phi(x) + \tau \rho(x) \right)  & \cdots & \partial^2_{{i_1} {j_d}} \! \left( \tilde \tau  \Phi(x)+ \tau \rho(x) \right)  \\
\vdots & \vdots & \ddots & \vdots \\
\partial_{{i_d}} \rho(x) \! \! & \partial^2_{{i_d} {j_1}} \! \left( \tilde \tau  \Phi(x) +  \tau \rho(x) \right)  & \cdots & \partial^2_{{i_d} {j_d}} \left( \tilde \tau \Phi(x) + \tau  \rho(x) \right) \\
0 & \partial_{{j_1}} \rho(x) & \cdots & \partial_{{j_d}} \rho(x)
\end{array} \right]  \right| \geq c \label{mainhyp} \end{equation}
Then for all $f_1,\ldots,f_{2d} \in L^2(\R)$,
\begin{equation} | I_\lambda (f_1,\ldots,f_{2d})| \lesssim |\lambda|^{-\frac{d-1}{2}} \prod_{j=1}^{2d} || f_j ||_2. \label{mainresult}
\end{equation}
\end{theorem}
The main hypothesis \eqref{mainhyp} shares features of both the usual H\"{o}rmander mixed Hessian condition as well as a Phong-Stein rotational curvature condition but is slightly different than either of these notions. It is also important to note that since the choice of partition can depend on the point $x$ and the pair $(\tilde \tau, \tau)$, the multilinear operator \eqref{myfunc} can be nondegenerate in situations where corresponding bilinear operators of several variables would be unavoidably singular. In particular,
examples of operators satisfying (\ref{mainhyp}) exist for all values of $d$ as we will demonstrate now. Using coordinates $(x_1,\ldots,x_d,x_1',\ldots,x_{d}')$ on $\R^{2d}$, take
\[ \rho(x) := x_1 x_1' + \cdots + x_{d-1} x_{d-1}' + x_1 + \cdots + x_d + x_1' + \cdots x_d'; \]
note $\rho$ has no second-order dependence on $x_d$ or $x_d'$, and satisfies (\ref{finiterhonorms}) and (\ref{implicit}) for appropriate constants on the set $B_1 = (-b_1, b_1)^{2d}$ for any positive $b_1 < 1$.
If $d$ is even, we select
$$ \Phi(x) := x_1x_{d/2+1} + x_1'x_{d/2+1}' + \cdots + x_{d/2}x_d,$$
where we have explicitly dropped the term $x_{d/2}'x_d'$ from the pattern,
and for odd $d$, we instead take
$$ \Phi(x) := x_1 x_{(d+1)/2} + x_1' x_{(d+1)/2}' + \cdots + x_{(d-1)/2}' x_{d-1}'.$$
For $d$ even, if $\tilde{\tau} \neq 0$, partition the variables into groups 
\begin{equation}
\{x_1, x_1', \dots, x_{d/2}, x_{d/2}' \} \; \text{ and } \; \{ x_{d/2+1}, x_{d/2+1}', \dots, x_d, x_d' \} 
\end{equation}
to conclude that
\begin{align*}
\det \left[\begin{array}{cccc} 
	\partial_{x_1} \rho & \tilde{\tau} \partial^2_{x_1 x_{d/2+1}} \Phi + \tau \partial^2_{x_1 x_{d/2+1}} \rho & \cdots & \tilde{\tau} \partial^2_{x_1 x_d'} \Phi + \tau \partial^2_{x_1 x_d'} \rho \\ 
	\vdots & \vdots & \ddots & \vdots \\ 
	\partial_{x_{d/2}'} \rho & \tilde{\tau} \partial^2_{x_{d/2}' x_{d/2+1}} \Phi + \tau \partial^2_{x_{d/2}' x_{d/2+1}}  \rho & \cdots & \tilde{\tau} \partial^2_{x_{d/2}' x_d'} \Phi + \tau \partial^2_{x_{d/2}' x_d'}  \rho \\
	0 & \partial_{x_{d/2+1}} \rho & \cdots & \partial_{x_d'} \rho \\ 
 \end{array} \right]
\end{align*} 
is exactly
\begin{align*}
\det \left[\begin{array}{cccccc} 
	1 + x_1' & \tilde{\tau} & 0 & \cdots & 0 & 0 \\
	1 + x_1 & 0 & \tilde{\tau} &   & \vdots &\vdots \\
	\vdots & \vdots &  & \ddots &0 & 0 \\
	1 + x_{d/2}' & 0 & \cdots & 0 & \; \tilde{\tau} \; & 0 \\
	1 + x_{d/2} & 0 & \cdots & & 0 & 0 \\
	0 & 1+x_{d/2+1}' & \cdots & 1+x_{d-1} & 1 & 1
 \end{array} \right] &= -(1+x_{d/2})\tilde{\tau}^{d-1}
\end{align*} 
whose magnitude is uniformly bounded below on $B_1$. If $\tau \neq 0$, partition the variables using the collections $\{x_1, \dots, x_d\}$ and $\{x_1', \dots, x_d'\}$ to see that the determinant corresponding to the condition \eqref{mainhyp} equals
\begin{align*}
\det \left[\begin{array}{cccccc} 
	1 + x_1' & \tau & 0 & \cdots & 0 & 0 \\
	\vdots & 0 & \tau &   & \vdots &\vdots \\
	\vdots & \vdots &  & \ddots &0 & 0 \\
	1 + x_{d-1}' & 0 & \cdots & 0 & \; \tau \; & 0 \\
	1  & 0 & \cdots & & 0 & 0 \\
	0 & 1+x_1 & \cdots &  & 1 + x_{d-1} & 1
 \end{array} \right] &= -\tau^{d-1} \neq 0
\end{align*} 
which is nonzero everywhere. If $d$ is odd, partitioning with either the grouping
\begin{equation*}
\{x_1, x_1', \dots, x_{(d-1)/2}, x_{(d-1)/2}', x_d \} \; \text{ and } \; \{x_{(d+1)/2}, x_{(d+1)/2}', \dots, x_{d-1}, x_{d-1}', x_d' \} 
\end{equation*}
or the grouping $\{x_1, \dots, x_d\}$ and $\{x_1', \dots, x_d'\}$ leads to a similar conclusion.

It is also meaningful to observe that the exponent $-(d-1)/2$ is best-possible for \eqref{mainresult}. To see this, suppose without loss of generality that $0 \in M$ and that $a$ is real and strictly positive on a neighborhood of $0$. Since $\partial_i \rho(x) \neq 0$ for all $i = 1, \dots, 2d$ and all $x \in B_1$, the Implicit Function Theorem gives (in particular) a function $\Psi$ and a neighborhood $\overline{0} \in U \subset \R^{2d-1}$ such that $M$ is the graph of $x_{2d} = \Psi(\overline{x})$ on $U$, where $\overline{x} := (x_1, \dots, x_{2d-1})$. By \eqref{finiterhonorms} and \eqref{implicit}, we have \begin{equation} |\Psi(\overline{x})| \leq c \max_{j \neq 2d} |x_j| \label{lip} \end{equation} for some positive $c$.
By Taylor's Theorem, $|\Phi(x) - \Phi(0) - \sum_{j=1}^{2d} x_j \partial_j \Phi(0)|$ is $O(|x|^2)$ as $x \to 0$, so there must be a sufficiently small $c' > 0$ such that for any $|\lambda| \geq 1$, $|x_j| < c'|\lambda|^{-1/2}$ for $j=1, \dots, 2d-1$ and $|x_{2d}| < cc'|\lambda|^{-1/2}$ for the $c$ from \eqref{lip} imply $|\lambda \Phi(x) - \lambda \Phi(0) - \lambda \sum_{j=1}^{2d} x_j \partial_j \Phi(0)| < \pi/4$, and therefore $\re[e^{i\lambda[\Phi(x) - \Phi(0) - \sum_j x_j \partial_j \Phi(0)]}] > 1/\sqrt{2}$. 
We can now find the desired lower bound for large $\lambda$ by testing $I_\lambda(f_1, \dots, f_{2d})$ on $f_j(x_j) := e^{-i\lambda x_j \partial_j \Phi(0)} \chi_{ \{ |\cdot| < c' |\lambda|^{-1/2} \}}(x_j)$, $j = 1, \dots, 2d-1$ and $f_{2d}(x_{2d}) := e^{-i\lambda x_{2d} \partial_{2d} \Phi(0)} \chi_{ \{ |\cdot| \leq cc' |\lambda|^{-1/2} \}}(x_{2d})$. Then
\begin{align*}
|I_\lambda&(f_1, \dots, f_{2d})| \geq |\re [ I_\lambda(f_1, \dots, f_{2d})] | \\
	& \geq C \left|  \int_M \prod_{j=1}^{2d-1} \chi_{|\cdot| < c' |\lambda|^{-1/2}}(x_j) \cdot \chi_{|\cdot| < c c' |\lambda|^{-1/2}}(x_{2d}) a(x) \, d\sigma(x) \right| \\
	& \geq C' \int_M \prod_{j=1}^{2d-1} \chi_{|\cdot| < c' |\lambda|^{-1/2}}(x_j) \cdot \chi_{|\cdot| < c c' |\lambda|^{-1/2}}(x_{2d}) \, d\sigma(x) 
\end{align*}
The second line follows from combining the terms $e^{i\lambda \Phi(x)}$, the complex exponential factors in the functions $f_j$, and the constant $1 = |e^{-i\lambda \Phi(0)}|$, and then using the estimate given by Taylor's Theorem (since all other factors in the integrand are real-valued).
We now may parametrize $M$ by $(\overline{x},\Psi(\overline{x}))$ and conclude that
\begin{align*}
|I_\lambda&(f_1, \dots, f_{2d})|  \\
	& \geq C'' \int_{\R^{2d-1}} \prod_{j=1}^{2d-1} \chi_{|\cdot| < c' |\lambda|^{-1/2}} (x_j) \cdot \chi_{|\cdot| < cc' |\lambda|^{-1/2}} (\Psi(\overline{x})) (1+|\nabla \Psi|^2)^{1/2} \, d\overline{x} \\
	& \geq C''' \int_{\R^{2d-1}} \prod_{j=1}^{2d-1} \chi_{|\cdot| < c' |\lambda|^{-1/2}} (x_j) \, d\overline{x} \geq C'''' |\lambda|^{-\frac{2d-1}{2}}.
\end{align*}
The third line follows, in part, from the fact that $|\Psi(\overline{x})| < cc'|\lambda|^{-1/2}$ on the support of the integrand, by \eqref{lip}. This shows sharpness because $|\lambda|^{-(2d-1)/2}$ is comparable to  $|\lambda|^{-(d-1)/2} \prod_j ||f_j||_2$.

The proof of Theorem \ref{mainthm} is accomplished in several stages. In Section \ref{decompsec}, we describe the decomposition that we use to study the multilinear operator \eqref{myfunc}. The decomposition features a scaling which in some sense interpolates between the standard Gabor and wavelet-type tilings of the time-frequency plane. We focus solely on the one-dimensional case that is of use to us but note that there are natural extensions to higher dimensions. Section \ref{spsec} contains a number of elementary lemmas related to stationary phase and integration on manifolds which will be used frequently throughout the proof. The proof of Theorem \ref{mainthm} is taken up fully in Section \ref{proofsec}, which is split into two parts: Section \ref{schwartzsec} deals with regions of rapid decay corresponding to so-called ``Schwartz tails'' and the main contributions are handled in Section \ref{mainsec}.

\section{A customized frequency space decomposition}
\label{decompsec}

Suppose $\lambda$ is any real number with $|\lambda| \geq 1$.  Let $n_0$ be the largest positive integer such that $n_0^2 \leq |\lambda|$ and let  $\mathcal Q_{\lambda}$ denote the collection of all intervals in $\R$ of one of the three following forms:
\begin{itemize}
\item $[k n_0, (k+1) n_0]$ for some $k \in \{-n_0,\ldots,n_0 - 1\}$,
\item $[n^2, (n+1)^2]$ for some integer $n \geq n_0$,
\item $[-(n+1)^2,-n^2]$ for some integer $n \geq n_0$.
\end{itemize}
The intervals in ${\mathcal Q}_\lambda$ are nonoverlapping and have lengths which grow like the square root of distance to the origin, in the sense that for any $\xi \in Q \in {\mathcal Q}_\lambda$,
\begin{equation} \frac{1}{9} |Q|^2 \leq  \max \left\{ |\lambda| ,  |\xi| \right\} \leq 4 |Q|^2. \label{boxsize} \end{equation}
Here and throughout, $|Q|$ denotes the length of $Q$. To see why \eqref{boxsize} holds, assume first that $Q$ is an interval of the first type. Then $|Q| = n_0$ and $|\xi| \leq n_0^2 \leq |\lambda|$, so 
\[ \frac{|Q|^2}{9} = \frac{n_0^2}{9} \leq  |\lambda| = \max\{  |\lambda|, |\xi| \}  \leq (n_0+1)^2 \leq 4 |Q|^2. \]
If $Q$ is of the second or third types, then $|Q| = 2n+1$ and $n^2 \leq |\xi| \leq (n+1)^2$ for some $n \geq n_0$. Then
\[ \frac{|Q|^2}{9} = \frac{(2n+1)^2}{9} \leq  n^2 \leq \max\{ |\lambda|, |\xi| \}  \leq (n+1)^2 \leq |Q|^2.\]

\begin{lemma}
Let $\varphi \in C^\infty(\R)$ be compactly supported with Fourier transform $\widehat \varphi$ strictly positive on $[-1/2,1/2]$. Let $\xi_Q$ be the center of $Q \in {\mathcal Q}_\lambda$ and let  \label{systemlemma}
\[ \varphi_Q(x) := |Q|^{\frac{1}{2}} e^{2 \pi i x \xi_Q} \varphi( |Q| x). \]
For all $\xi \in \R$ in the interior of $Q \in {\mathcal Q}_\lambda$, fix $\varphi_{\xi} := \varphi_Q$, and let $\varphi_\xi := 0$ if $\xi$ is a boundary point of any $Q \in {\mathcal Q}_\lambda$. There is a bounded linear map $V$ from $L^2(\R)$ into $L^2(\R \times \R)$ and a dense subspace of $L^2(\R)$ such that
\begin{equation} f(x) = \int_{\R \times \R} V f(y,\xi) \varphi_{\xi} (x-y) ~dy d \xi \label{expansion} \end{equation}
for all $f$ in the dense subspace, with the integrand belonging to $L^1(\R \times \R)$ for all such $f$. The norm of $V$ depends only on the choice of $\varphi$ (and in particular is independent of $\lambda$).
\end{lemma}
\begin{proof}
We fix the dense subspace to be all $f$ whose Fourier transform is compactly supported.
If $f_Q$ denotes Fourier projection onto the interval $Q$, i.e., $\widehat f_Q := \widehat f \chi_Q$, then clearly
\[ f = \sum_{Q \in \mathcal Q_\lambda} f_Q \]
for all $f$ in the dense subspace, with no convergence issues because all but finitely many terms are zero.
Let $T_Q$ be the bounded operator on $L^2(\R)$ which satisfies 
\[ \widehat{T_Q f} (\xi) := \widehat {f} (\xi) \frac{\chi_{Q}(\xi)}{\widehat{\varphi}(|Q|^{-1}(\xi - \xi_Q))} \]
for all $f \in L^2(\R)$; $\widehat{\varphi}$ being bounded below on $[-1/2,1/2]$ is precisely what makes $T_Q$ bounded on $L^2(\R)$. The $T_Q$ are mutually orthogonal, uniformly bounded operators satisfying
\[ \int (T_Q f)(y) \varphi_Q(x-y) dy = |Q|^{-\frac{1}{2}} f_Q(x) \]
for every $Q \in {\mathcal Q}_\lambda$ and every $f \in L^2(\R)$ (note: the integrand is in $L^1$ because $T_Q f$ and $\varphi_Q$ belong to $L^2(\R)$). For any $\xi \in \R$ in the interior of a $Q \in \mathcal Q_\lambda$, let
\[ V f (y, \xi) := |Q|^{-1/2} (T_Q f)(y). \]
The function $V f(y,\xi)$ is constant as a function of $\xi$ on the interior of any $Q \in {\mathcal Q}_\lambda$ and is square integrable as a function of $y$.
The identity
 \[ f(x) = \int_{\R \times \R} V f(y,\xi) \varphi_\xi (x-y) ~ dy d \xi \]
must hold because $\hat f$ is compactly supported, which puts the integrand in $L^1(\R \times \R)$ and permits decomposing the $\xi$ integral into a sum over the $Q \in {\mathcal Q}_\lambda$, i.e., 
\begin{align*}
  \int_{\R \times \R} & V f(y,\xi) \varphi_\xi (x-y) ~ dy d \xi = \sum_{Q \in \mathcal Q_\lambda} |Q| \int |Q|^{-1/2} (T_Q f)(y) \varphi_Q(x-y) dy \\
 &  = \sum_{Q \in \mathcal Q_\lambda} f_Q(x) = f(x). 
  \end{align*}
Finally, observe that the same decomposition of the $\xi$ integral gives
\[ \int_{\R^d \times \R^d} |V f(y,\xi)|^2 dy d \xi = \sum_{Q \in \mathcal Q_\lambda} |Q| \int \left| |Q|^{-1/2} (T_Q f)(y) \right|^2 dy \leq C ||f||^2_2 \]
where $C := \sup_{|\xi| \leq 1/2} |\widehat \varphi(\xi)|^{-1}$ (which is independent of $\lambda$).
\end{proof}

The functions $\varphi_\xi(x)$ satisfy a differential inequality which is of the sort found in H\"{o}rmander's exotic symbol class $S^0_{1/2,1/2}$, which in some sense exhibits a kind of scaling which falls between the windowed Fourier transform and a Littlewood-Paley decomposition. In our case, the important feature here is that derivatives with respect to $x$ lose roughly a factor of $|\xi|^{1/2}$:
\begin{proposition}
Suppose that, in addition to the hypotheses of Lemma \ref{systemlemma}, $\varphi$ is supported on $[-1/4,1/4]$.  \label{wavepacket2}
For any $\xi \in \R$, let $r := (\max \{ |\lambda|, |\xi| \})^{-1/2}$. For any $k \geq 0$, there is a constant $C_{k}$ depending only on $k$ and $\varphi$ such that
\begin{equation} \left|  \partial_x^k \left( e^{- 2 \pi i x \xi} \varphi_{\xi} (x) \right) \right| \leq \begin{cases} C_{k} r^{-\frac{1}{2} - k}  & x \in \left[ - \frac{r}{2}, \frac{r}{2} \right]  \\
0 & \mbox{ otherwise} \end{cases}. \label{derivest} \end{equation}
\end{proposition}
\begin{proof}
By definition of $\varphi_\xi(x)$, it must vanish identically when $|x| \geq (4 |Q|)^{-1}$ for some interval $Q \in {\mathcal Q}_\lambda$ containing $\xi$, and by \eqref{boxsize}, $|Q|^{-1} \leq 2 r$, so $|x| \leq r/2$ is necessary for the left-hand side of \eqref{derivest} to be nonzero.
By the product rule,
\begin{align*}
|Q|^{\frac{1}{2}} &\partial_x^k \left( e^{2 \pi i (\xi_Q - \xi) x} \varphi ( |Q| x) \right) \\
& = |Q|^{\frac{1}{2} + k} \sum_{0 \leq \beta \leq k} c_{k,\beta} \left( \frac{2 \pi i (\xi_Q - \xi) }{|Q|} \right)^{\beta} e^{2\pi i (\xi_Q - \xi)x}  (\partial_x^{k-\beta} \varphi)( |Q| x) 
\end{align*}
for some constants $c_{k,\beta}$ depending only on $k$ and $\beta$. The desired inequality \eqref{derivest} follows from this identity and the triangle inequality
because $|\xi_Q - \xi| \leq |Q|$, as both $\xi$ and $\xi_Q$ belong to $Q$, and because $|Q| \leq 3 r^{-1}$ by \eqref{boxsize}.
\end{proof}
Interestingly, while there are several authors who have studied such scaling as it applies to symbols (see Stein \cite{steinha} for a more classical exposition;  Beltran and Bennett \cite{bb2017} and Beltran \cite{beltran2019} relate this kind of scaling to novel geometric maximal function inequalities), there do not appear to be any previous instances of a decomposition of this sort being applied to general phases $\Phi$ or in geometric settings as we have here. In Section \ref{proofsec} we will see that the decomposition (regarding $V f$ as an analysis operator) is so efficient that it essentially diagonalizes \eqref{myfunc}---at no point do we even need a $TT^*$ argument or to employ orthogonality of any of the various terms we encounter. This strongly suggests that a similar decomposition could greatly simplify many existing proofs for things like Sobolev estimates for Fourier Integral Operators.

Another natural point is to consider whether it is possible to construct a ``nice'' discrete system which in some way mirrors the decomposition, e.g., whether it is possible to construct a family of functions $\{\psi_n\}_{n=1}^{\infty}$ whose Fourier transforms $\widehat \psi_n$ are adapted to boxes $[-n/2,n/2]^d$ such that the functions $\psi_{Q,k} := e^{2 \pi i \xi_{Q} \cdot x} \psi_n(x - n^{-1}k)$ (where $n$ is taken to be the side length of $Q$ and $k \in \Z^d$) form a Parseval frame for $L^2(\Z^d)$ as $Q$ ranges over $\mathcal Q_\lambda$ and $k$ ranges over $\Z^d$.  Taking $\widehat \psi_n := \chi_{[-n/2,n/2]^d}$ accomplishes this task, but the functions $\psi_n$ are poorly localized in space despite being highly localized in frequency. One can see without difficulty, however, that under the structural constraints just suggested, it is never be possible for all the $\psi_n$ to be compactly supported. Using the characterization of discrete shift-invariant systems by Hern\'andez, Labate, and Weiss \cite{hlw2002}, one can verify that when $n_0$ is the smallest side length of a box in $\mathcal Q_{\lambda}$, the necessary condition
\[ \sum_{Q \in {\mathcal Q}_\lambda, |Q| = n_0^d} \widehat \psi_{n_0}(\xi + \xi_Q) \overline{\widehat \psi_{n_0}(\xi + \xi_Q + (n_0,\ldots,n_0))} = 0 \mbox{ for almost every } \xi \] 
forces the summand to be zero almost everywhere because one may iteratively use the sum as a linear recurrence relation for a sum over a lower-dimensional product set and observe that functions satisfying linear recurrence relations for shifts must be zero almost everywhere if they are integrable.
There are various simple ways that one may loosen the structural constraints (e.g., by not insisting that boxes of the same size be represented by frequency shifts of the same function and by mildly oversampling translations) to construct a system of functions $\psi_{Q,k}$ which are all Schwartz functions. It turns out however, that for the present purposes it is vastly simpler to analyze the functional \eqref{myfunc} using the continuous system from Lemma \ref{systemlemma} than it would be to use any analogous discrete system, even if one could find such a discrete system with no ``Schwartz tails.''

\section{Stationary Phase}
\label{spsec}

We turn now to the task of proving a few useful inequalities which will be applied repeatedly in the next section. The first is a variation on the standard stationary phase estimate:
\begin{lemma}
For each positive integer $N$, there is a constant $C_N$ such that for every smooth manifold $M$ with measure $d \sigma$ of smooth positive density, every pair $(\varphi,\psi)$ of $C^N$ real-valued functions on $M$ with $\psi$ compactly supported, and every nonzero complex number $K$, \label{ibplemma}
\begin{equation} \left| \int e^{i \varphi} \psi d \sigma \right| \leq \frac{C_N}{|K|^N} \sum_{j = 0}^N \int \left| (X^*)^{j} \psi \right| \left[ | X \varphi - K |^{N - j
} + \sum_{\ell=2}^{N-j} | X^{\ell} \varphi |^{\frac{N-j}{\ell}} \right] d \sigma, \label{mainibp} \end{equation}
where $X$ is any $C^N$ vector field on $M$ and $X^*$ is the first-order differential operator dual to $X$.
\end{lemma}
\begin{proof}
By definition of $X^*$, the chain rule, and integration by parts,
\[ \int (i X \varphi) e^{i \varphi} \psi d \sigma = \int \left( X e^{i \varphi} \right) \psi d \sigma = \int e^{i \varphi} (X^* \psi) d \sigma. \]
As a consequence, for any nonzero complex number $\tilde K$,
\[ \int e^{i \varphi} \psi d \sigma = \frac{1}{\tilde K} \int e^{i \varphi} \left[ X^* \psi - (i X \varphi - \tilde K) \psi \right] d \sigma . \]
Consider the differential operator $L \psi := \left[ X^* \psi - (i X \varphi - \tilde K) \psi \right]$.
By a repeated application of the above computation, it must be the case that
\begin{equation} \int e^{i \varphi} \psi d \sigma = \tilde K^{-N} \int e^{i \varphi} L^N \psi ~ d \sigma \label{mainibp0} \end{equation}
for any positive $N$ and $C^N$ functions $\varphi$ and $\psi$. 
Let $A_N \subset \Z_{\geq 0}^{N+1}$ be the set of multiindices $\alpha := (\alpha_0,\ldots,\alpha_N)$ such that $\alpha_0 + \sum_{\ell=1}^N \ell \alpha_\ell = N$.   By induction on $N$, we claim that
\begin{equation} L^N \psi = \sum_{\alpha \in A_N} C_{N,\alpha} ((X^{*})^{\alpha_0} \psi ) (i X \varphi - \tilde K)^{\alpha_1} \prod_{\ell=2}^{N} (X^\ell \varphi)^{\alpha_\ell} \label{theclaim} \end{equation}
for some constants $C_{N,\alpha}$ depending only on $N$ and $\alpha$. For convenience, given any finite sequence $\alpha := (\alpha_0,\ldots,\alpha_{N'})$, let
\[ T_\alpha := ((X^{*})^{\alpha_0} \psi ) (i X \varphi - \tilde K)^{\alpha_1} \prod_{\ell=2}^{N'} (X^\ell \varphi)^{\alpha_\ell}. \]
To establish the claim, it suffices to show that the application of $L$ to any $T_\alpha$ with $\alpha \in A_N$ yields a linear combination of terms $T_\beta$ with $\beta \in A_{N+1}$ in such a way that the coefficients depend only on $\alpha$.  First, the product rule gives that
\begin{align*} X^*  T_\alpha  = & ~ ((X^*)^{\alpha_0 + 1} \psi ) (i X \varphi - \tilde K)^{\alpha_1} \prod_{\ell=2}^{N} (X^\ell \varphi)^{\alpha_j} \\ & -  
 i \alpha_1 ((X^{*})^{\alpha_0} \psi )  ( i X \varphi - \tilde K)^{\alpha_1 - 1} (X^2 \varphi) \prod_{\ell=2}^{N} (X^\ell \varphi)^{\alpha_\ell} \\ & - ((X^{*})^{\alpha_0} \psi )  (i X \varphi - \tilde K)^{\alpha_1} \sum_{\ell'=2}^{N} \alpha_{\ell'} (X^{\ell'+1} \varphi) (X^{\ell'} \varphi)^{\alpha_{\ell'} - 1} \mathop{\prod_{\ell=2}^N}_{\ell \neq \ell'} (X^\ell \varphi)^{\alpha_\ell} 
 \end{align*}
 since $X^* (fg) = (X^* f) g - f (Xg)$ for any smooth functions $f$ and $g$.
Each term on the right-hand side is exactly a term $T_\beta$ for $\beta \in A_{N+1}$ {(because the terms on the second and third lines have the exponent $\alpha_\ell$ decreased by $1$ and the exponent $\alpha_{\ell + 1}$ increased by $1$ for some $\ell \geq 1$)} and the coefficients clearly only depend on $\alpha$.   Likewise $(i X \varphi - \tilde K) T_{\alpha}$ is equal to $T_{\beta}$ for some $\beta \in A_{N+1}$. This establishes \eqref{theclaim}. The passage to \eqref{mainibp} follows from \eqref{mainibp0} and \eqref{theclaim} by the triangle inequality and the inequality for arithmetic and geometric means:
\begin{align*}
\left| \int e^{ i \varphi} \psi d \sigma \right|   \leq  & \sum_{\alpha \in A_N} \frac{|C_{N,\alpha}|}{|\tilde K|^N} \int | (X^*)^{\alpha_0} \psi| | i X \varphi  - \tilde K|^{\alpha_1} \prod_{\ell=2}^N |X^\ell \varphi|^{\alpha_\ell} d \sigma \\
 \leq  \sum_{\alpha \in A_N}  \frac{|C_{N,\alpha}|}{|\tilde K|^N} & \int \frac{| (X^*)^{\alpha_0} \psi|}{N - \alpha_0}  \left( \alpha_1 | i X \varphi  - \tilde K|^{N - \alpha_0 } + \sum_{\ell=2}^N \ell \alpha_\ell |X^\ell \varphi|^{\frac{N - \alpha_0}{\ell} } \right) d \sigma.
\end{align*}
Finally, \eqref{mainibp} follows by consolidating like terms and setting $\tilde K := i K$.
\end{proof}

The next result complements the stationary phase lemma and applies primarily to non-oscillatory integrands. It is needed because the trivial ``size'' inequalities one typically sees are formulated on $\R^n$ rather than on manifolds $M$:
\begin{proposition}
Suppose $M \subset B_1 = (-b_1, b_1)^{2d}$ is the zero set of a function $\rho$ which satisfies the inequalities \eqref{finiterhonorms} and \eqref{implicit} and that $f$ is a measurable function on $B_1$ supported in a product of intervals $I := I_1 \times \cdots \times I_{2d}$. Then for any $j_0 \in \{1,\ldots,2d\}$,\label{iftprop}
\begin{equation} \left| \int_M f ~ d \sigma \right| \lesssim ||f||_{L^\infty(M)} \mathop{\prod_{j =1}^{2d}}_{j \neq j_0} |I_j| \label{ift} \end{equation}
where $d \sigma$ is Lebesgue measure on $M$.
\end{proposition}
\begin{proof}
By virtue of the assumption \eqref{implicit}, the manifold $M$ intersects any line parallel to a coordinate axis in $\R^{2d}$ at most once, which means that for any $j_0$, $M$ is the graph of some function $x_{j_0} = \Psi_{j_0}(x_1,\ldots,\widehat{x_{j_0}},\ldots,x_{2d})$ ($\widehat{\cdot}$ denotes omission) on $U_{j_0} := (-b_1,b_1)^{2d-1}$. By \eqref{finiterhonorms} and \eqref{implicit}, the derivatives $\partial_{j} \Psi_{j_0}$ have magnitude at most $C_\rho C_\rho'$. 
Since the Radon-Nikodym derivative of Lebesgue measure $d \sigma$ on the graph with respect to Lebesgue measue in coordinates $(x_1,\ldots,\widehat{x_{j_0}},\ldots,x_{2d})$ is exactly $(1 + |\nabla \Psi_{j_0}|^2)^{1/2}$, it follows that any measurable function $f$ supported on a product of intervals $I_1 \times \cdots \times I_{2d} \subset B_1$ will satisfy
\begin{align*}
\left| \int_M f d \sigma \right| & = \left| \int_{U_{j_0}} (f \circ \Psi_{j_0})  (1 + |\nabla \Psi_{j_0}|^2)^{1/2} dx_{1} \cdots \widehat{ d x_{j_0}} \cdots d x_{2d}  \right| \\
 \leq  \int_{U_{j_0}}&  ||f||_{L^\infty(M)} \prod_{j=1}^{2d} \chi_{I_j} (x_j) (1 + (2d-1) (C_\rho C_\rho')^2)^{\frac{1}{2}} dx_{1} \cdots \widehat{ d x_{j_0}} \cdots d x_{2d},
\end{align*}
which gives exactly \eqref{ift} by Fubini.
\end{proof}
We conclude the section by stating the following corollary, which combines our size and oscillation estimates:
\begin{corollary}
For $M$ as in Proposition \ref{iftprop} and $\varphi, \psi$, and $X$ as in Lemma \ref{ibplemma}, if $\psi$ is supported on a product of intervals $I := I_1 \times \cdots \times I_{2d} \subset B_1$ and if $N$ is any fixed positive integer, then
\begin{equation} 
\begin{split} 
\left| \int_M e^{i \varphi} \psi d \sigma \right| \lesssim \frac{C_N |I|}{|K|^N |I_{j_0}|} \sum_{k=0}^N || (X^*)^k \psi ||_{L^\infty(M \cap I)} &  \left[ \vphantom{\sum_{\ell=2}^{N-k}} \right.  || X \varphi - K ||_{L^\infty(M \cap I)}^{N-k}  \\
 & \left. + \sum_{\ell=2}^{N-k} ||X^\ell \varphi ||_{L^\infty(M \cap I)}^{\frac{N-k}{\ell}} \right]
 \end{split}  \label{fullibp}
\end{equation}
for any complex number $K$ and any $j_0 \in \{1,\ldots,2d\}$.
\end{corollary}

\section{Proof of Theorem \ref{mainthm}}
\label{proofsec}
\subsection{Main kernel estimate and cases of rapid decay}
\label{schwartzsec}

To estimate the quantity $I_\lambda(f_1,\ldots,f_{2d})$ defined by \eqref{myfunc}, each function $f_j$ will be expressed via the expansion \eqref{expansion}. To that end, fix $\varphi$ to be a function of one real variable satisfying the hypotheses of Lemma \ref{systemlemma} and Proposition \ref{wavepacket2}. Using \eqref{expansion}, it suffices to study
\begin{equation} \mathcal I(y,\xi) := \int_M e^{i \lambda \Phi(x)} \left[ \prod_{j=1}^{2d} \varphi_{\xi_j}(x_j - y_j) \right] a(x) d \sigma(x) \label{kerneldef} \end{equation}
on $\R^{2d} \times \R^{2d}$ and prove an inequality of the form
\[ \left| \int_{(\R \times \R)^{2d}} \mathcal I(y,\xi) \prod_{j=1}^{2d} f_j(y_j,\xi_j) dy_1 d \xi_1 \cdots d y_{2d} d \xi_{2d} \right| \lesssim |\lambda|^{-\frac{d-1}{2}} \prod_{j=1}^{2d} ||f_j||_{L^2(\R \times \R)} \]
for arbitrary $f_1,\ldots,f_{2d} \in L^2(\R \times \R)$.

For each $\xi \in \R^{2d}$, let $r_j := (\max \{ |\lambda|, |\xi_j| \})^{-1/2}$ and $r := \min \{r_1,\ldots,r_{2d}\}$. Because $\varphi$ satisfies the hypotheses of Proposition \ref{wavepacket2}, each $\varphi_{\xi_j}$ will be supported in the interval $[-r_j/2,r_j/2]$. Points $x \in M$ may always be assumed to belong to the support of the amplitude function $a \subset B_0$, so the constraint \eqref{lambdaisntsmall} and the inequalities $r_j \leq \lambda^{-1/2}$ imply that $\mathcal I(y,\xi) = 0$ when $y \not \in B_1$.

For each $i =1,\ldots,2d$, let $X_i$ be the smooth vector field on $M$ given by
\begin{equation} X_i := \partial_{i} - \frac{\partial_{i} \rho}{| \nabla \rho|^2} \sum_{j=1}^{2d} (\partial_{j} \rho) \partial_{j}. \label{thevec} \end{equation}
The central calculation in the proof of Theorem \ref{mainthm} is to apply Lemma \ref{ibplemma} to estimate the magnitude of $\mathcal I(y,\xi)$:
\begin{proposition}
Let $N \leq 2d+2$ be a positive integer. Assuming  \eqref{finiterhonorms}, \eqref{finitephinorms}, \eqref{finiteanorms},  \eqref{implicit}, and \eqref{lambdaisntsmall}, it follows that \label{firstsize}
\begin{equation} |\mathcal I (y ,\xi) | \lesssim \left( 1 + \frac{r^2}{\max_j r_j} \left( \sum_{i=1}^{2d} |X_i ( \lambda \Phi + 2 \pi \xi \cdot x)|_{y}|^2 \right)^{\frac{1}{2}}   \right)^{-N}  r_{j_0}^{-1} \prod_{j=1}^{2d} r_j^{\frac{1}{2}} \label{sizeest} \end{equation}
for every $j_0 \in \{1,\ldots,2d\}$. Moreover, if ${\mathcal I}(y,\xi) \neq 0$ then $|\rho(y)| \lesssim \max_j r_j$.
\end{proposition}
\begin{proof}
First observe that \eqref{ift} implies the inequality
\begin{equation} \int_M \prod_{j=1}^{2d} | \varphi_{\xi_j}(x_j - y_j)| |a(x)| d \sigma(x) \lesssim ||a||_{L^\infty} r_{j_0}^{-\frac{1}{2}}  \prod_{j \neq j_0} r_j^{\frac{1}{2}} \label{support} \end{equation}
since $\varphi_{\xi_j}$ is supported on an interval of length $r_j$ and has $|| \varphi_{\xi_j}||_\infty \lesssim r_j^{-1/2}$.
For convenience, write $\mathcal I(y,\xi)$ as
\[ {\mathcal I}(y,\xi) = \int_{M} e^{i (\lambda \Phi(x) + 2 \pi \xi \cdot x)}  \left[ e^{- 2 \pi i \xi \cdot y} a(x) \prod_{j=1}^{2d}  e^{- 2 \pi i \xi_j (x_j - y_j)} \varphi_{\xi_j}(x_j - y_j) \right]   d \sigma(x) \]
and let $a_{y,\xi}(x)$ denote the function in brackets.
Let $\tilde r := r^2 / \max_j r_j$ and apply \eqref{fullibp} with $\psi := a_{y,\xi}$,  $\varphi := \lambda \Phi(x) + 2 \pi \xi \cdot x$, and  $X := \tilde r X_i$.
It follows that
\begin{equation}
\begin{split}
 |\mathcal I(y,\xi)| \lesssim \frac{ \prod_{j} r_j }{|K|^N r_{j_0}}  \sum_{k=0}^N ||(\tilde r X^*_i)^k&  a_{y,\xi}||_{\infty}  \left[ \vphantom{\sum_{\ell=2}^{N-k}} \right.  || \tilde r X_i ( \lambda \Phi + 2 \pi  \xi \cdot x) - K||_{\infty} ^{N-k} \\ & \left. + \sum_{\ell=2}^{N-k} ||(\tilde rX_i)^\ell ( \lambda \Phi + 2 \pi \xi \cdot x)||_{\infty}^{\frac{N-\ell}{k}} \right] 
\end{split} \label{applyibp}
\end{equation}
for any nonzero choice of $K$, where the $L^\infty$ norm is taken on $M \cap I_{y,\xi}$ for $I_{y,\xi} := y + [-r_1/2,r_1/2] \times \cdots \times [-r_{2d}/2,r_{2d}/2]$.  To estimate the right-hand side, begin by considering the action of $X_i^*$ on an arbitrary function $f$:
\begin{align*}
 X_i^* f  
 = - \partial_i f + \sum_{j=1}^n \frac{\partial_i \rho \partial_j \rho}{|\nabla \rho|^2} \partial_j f + f \sum_{j=1}^{2d} \partial_j \left( \frac{\partial_i \rho \partial_j \rho}{|\nabla \rho|^2}  \right).
 \end{align*}
From this formula and the quotient rule, one sees that $X_i^* f$ may be expressed as a linear combination of $f$ and its first coordinate derivatives with coefficients which are polynomial functions of $|\nabla \rho|^{-2}$ and the first and second derivatives of $\rho$. By induction on $k \leq N$, this means that $(X_i^*)^k f$ will be a linear combination of the derivatives $\partial^{\alpha} f$ for $|\alpha| \leq k$ with coefficients which are polynomials in $|\nabla \rho|^{-2}$ and the derivatives $\partial^{\beta} \rho$ for $1 \leq |\beta| \leq k + 1$ (the maximum order of differentiation of $\rho$ grows like $k+1$ because the coefficients themselves are differentiated at most one time for every subsequent application of an $X_i^*$). It follows from the inequalities \eqref{derivest} and \eqref{finiteanorms} as well as the bounds \eqref{finiterhonorms} and \eqref{implicit} on derivatives of $\rho$ that for each $k \in \{0,\ldots,N\}$,
\begin{equation} |(\tilde r X^*_i)^k a_{y,\xi}| \lesssim \prod_{j=1}^{2d} r_j^{-\frac{1}{2}}. \label{nlinderivest} \end{equation}
Here the factor of $\tilde r^k$ on the left-hand side compensates for the powers of $r_j^{-1}$ in \eqref{derivest} that grow with the number of derivatives because $\tilde r r_j^{-1} \leq r r_j^{-1} \leq 1$ for all $j$. Terms with fewer than $k$ derivatives falling on the modulated $\varphi_{\xi_j}$ are controlled by the same upper bound because $\tilde r$ is bounded above by virtue of \eqref{lambdaisntsmall}.
A similar argument gives
\[ \sum_{\ell=2}^{N-k} |(\tilde rX_i)^\ell ( \lambda \Phi + 2 \pi \xi \cdot x)|^{\frac{N-\ell}{k}} \lesssim 1 \]
because \eqref{lambdaisntsmall} and $\ell \geq 2$  imply that $\tilde r^{\ell} (| \lambda| + |\xi|) \lesssim 1$ and because the derivatives of $\Phi$ are bounded by \eqref{finitephinorms}.
Therefore \eqref{applyibp} implies
\[ | {\mathcal I} (y,\xi)| \lesssim \frac{1}{|K|^N} \left( r_{j_0}^{-\frac{1}{2}} \prod_{j \neq j_0} r_j^{\frac{1}{2}} \right) \! \! \left( 1 + \sum_{k=0}^N \sup_{x \in I_{y,\xi}} | \tilde r X_i ( \lambda \Phi + 2 \pi  \xi \cdot x) - K|^{N-k} \right)  \]
for any choice of $i$ and $j_0$.
For each triple $(y,\xi,\lambda)$ and each $i=1,\ldots,2d$, let $K_i := \left. \tilde r X_i (\lambda \Phi + 2 \pi \xi \cdot x) \right|_y$. Then let $K := K_i$ for any index $i$ maximizing $|K_i|$.  If $|K| \leq 1$, then \eqref{sizeest} follows immediately from \eqref{support}, so assume $|K| \geq 1$. For any $x$ in the support of $a_{y,\xi}$, 
\[ \sup_{x \in I_{y,\xi}} \left| \tilde r  \left.  X_i ( \lambda \Phi + 2 \pi  \xi \cdot x) \right|_x - \tilde r  \left.  X_i ( \lambda \Phi + 2 \pi  \xi \cdot x) \right|_y \right| \lesssim 1. \]
This follows from the Mean Value Theorem via the inequalities
\begin{align*}
 \left|  \left. \lambda \tilde r X_i \Phi \right|_x -  \left. \lambda \tilde r X_i \Phi \right|_y\right| \leq |\lambda| \tilde r \sum_{j=1}^{2d} |x_j - y_j| \sup_{z \in B_1} | \partial_j X_i \Phi(z)| \lesssim |\lambda| \tilde r \max_{j} r_j = |\lambda| r^2
 \end{align*}
 and
 \[ 2 \pi \tilde r \left| \left( \left. \xi_i -  \frac{\partial_i \rho}{|\nabla \rho|^2} \nabla \rho \cdot \xi  \right|_x \right) -  \left( \left. \xi_i -  \frac{\partial_i \rho}{|\nabla \rho|^2} \nabla \rho \cdot \xi  \right|_y \right) \right| \lesssim \tilde r |\xi| \max_j r_j = r^2 |\xi| \]
 because $|\lambda| r^2 \leq 1$ and $r^2 |\xi| \leq \sqrt{2d}$.
 Thus it must be the case that $ | {\mathcal I} (y,\xi)| \lesssim  |K|^{-N} r_{j_0}^{-1/2} \prod_{j \neq j_0} r_j^{1/2}$. Because $|K| \geq 1$ and $|K| \geq |K_i|$ for each $i$,
 \[ 1 + \tilde r \left( \sum_{j=1}^{2d} \left|\left. X_i(\lambda \Phi + 2 \pi \xi \cdot x) \right|_y \right|^2 \right)^{\frac{1}{2}} \leq (1 + \sqrt{2d}) |K|, \]
 so \eqref{sizeest} must continue to hold when $|K| \geq 1$ as well as a direct consequence of the stronger inequality $ | {\mathcal I} (y,\xi)| \lesssim  |K|^{-N} r_{j_0}^{-1/2} \prod_{j \neq j_0} r_j^{1/2}$.
 
To finish, note that $\mathcal I(y,\xi) = 0$ unless $|y_j - x_j| \leq r_j$ for each $j$ and some $x \in M$, which 
means that
$|\rho(y)| = |\rho(y) - \rho(x)| \leq (\max_j r_j) \sup_{z \in B_1} \sum_{i=1}^{2d} |\partial_i \rho(z)|$ by the Mean Value Theorem.
\end{proof}

It is helpful to explicitly compute the action of the vector fields $X_i$. Since
\begin{equation} X_i ( \lambda \Phi + 2 \pi \xi \cdot x)|_{y} = \lambda \partial_i \Phi(y) + 2 \pi \xi_i - \frac{\partial_i \rho(y)}{|\nabla \rho(y)|^2} ( \lambda \nabla \Phi (y) + 2 \pi \xi) \cdot \nabla \rho(y) \label{vcalc} \end{equation}
for each $i$, 
the vector $(X_1 ( \lambda \Phi + 2 \pi \xi \cdot x)|_{y},\ldots,X_{2d} ( \lambda \Phi + 2 \pi \xi \cdot x)|_{y})$ is merely the orthogonal projection of $\lambda \nabla \Phi(y) + 2 \pi \xi$ onto the space orthogonal to $\nabla \rho(y)$.

It is also convenient to break the analysis  of ${\mathcal I}(y,\xi)$ into two regions depending on $\xi$. Letting $c$ be some small positive constant, the first region will be those $\xi$ for which $\min_{j=1,\ldots,2d} r_j \leq c \max_{j=1,\ldots,2d} r_j$.  The second region will be the remaining part: $\min_{j =1,\ldots,2d} r_j \geq c \max_{j=1,\ldots,2d}$. The analysis of $\mathcal I(y,\xi)$ is much more subtle on the second region than the first, but the second region has the small advantage that all the lengths $r_1,\ldots,r_{2d}$ are always comparable on this region, meaning that $r_j \approx r \approx \min \{ |\lambda|^{-1/2}, |\xi|^{-1/2} \}$. The proposition below gives a minor refinement of this splitting and establishes good boundedness properties of ${\mathcal I}(y,\xi)$ on the first region and a related set:
\begin{proposition}
Let $\Xi_1$ be the set consisting of all $\xi \in \R^{2d}$ such that $$\min_{j=1,\ldots,2d} r_j \leq c \max_{j=1,\ldots,2d} r_j.$$ For all sufficiently small $c > 0$ depending only on admissible constants and any fixed value of $N \leq 2d+2$, every pair $(y,\xi) \in B_1 \times \Xi_1$ satisfies the inequality \label{goodpart}
\begin{equation} \left| {\mathcal I}(y,\xi) \right| \lesssim ( |\lambda| + |\xi| )^{-(N-1)/2} |\lambda|^{-\frac{d}{2}}.\label{rapiddecay1} \end{equation}
Moreover, for all sufficiently small $c' > 0$ and all sufficiently large $c''$ depending only on admissible constants, if $(y,\xi) \in B_1 \times (\R^{2d} \setminus \Xi_1)$ satisfies $|\xi| \geq c'' |\lambda|$ and 
\begin{equation} \left| \frac{( \lambda \nabla \Phi (y) + 2 \pi \xi) \cdot \nabla \rho(y)}{|\nabla \rho(y)|^2} \right| \leq c' r^{-2}, \label{smalltao} \end{equation}
then
\begin{equation} 
|\mathcal I(y,\xi)| \lesssim (|\lambda| + |\xi|)^{-\frac{N}{2}} |\lambda|^{-\frac{d-1}{2}}. \label{rapiddecay2} 
\end{equation}
If $E \subset \R^{2d} \times \R^{2d}$ consists of all $(y,\xi)$ to which either \eqref{rapiddecay1} or \eqref{rapiddecay2} applies, then
\begin{equation}
\begin{split}
\left| \int_{\R^{2d} \times \R^{2d}} \! \! {\mathcal I}(y,\xi) \chi_{E}(y,\xi) \prod_{j=1}^{2d} f_j(y_j,\xi_j) dy d \xi \right|  \lesssim |\lambda|^{-\frac{d-1}{2}} \prod_{j=1}^{2d} ||f_j||_{L^2}
\end{split} \label{schwartzpart}
\end{equation}
for all functions $f_j \in L^2(\R \times \R)$.
\end{proposition}
\begin{proof}
We first show \eqref{rapiddecay2}. In this case $r_j \approx r$, which along with \eqref{vcalc} gives
\[ \tilde r \left(\sum_{i} |X_i (\lambda \Phi + 2 \pi i \xi \cdot x)|_y|^2\right)^{1/2} \approx r \left| \lambda \nabla \Phi + 2 \pi \xi - \nabla \rho \frac{(\lambda \nabla \Phi + 2\pi \xi) \cdot \nabla \rho}{|\nabla \rho|^2} \right|. \]
The assumptions $|\xi| \geq c'' |\lambda|$ and \eqref{smalltao} along with the triangle inequality imply
\[ r \left| \lambda \nabla \Phi + 2 \pi \xi - \nabla \rho \frac{(\lambda \nabla \Phi + 2\pi \xi) \cdot \nabla \rho}{|\nabla \rho|^2} \right| \gtrsim r \left[ 2 \pi |\xi| - \frac{|\xi|}{c''}|\nabla \Phi| - |\nabla \rho| c' r^{-2} \right]. \]
Since $r^{-2} \approx |\xi|$ when all $r_j$ are comparable and $|\xi| \geq c'' |\lambda|$, if $c''$ is sufficiently large and $c'$ is sufficiently small depending on \eqref{finitephinorms} and \eqref{finiterhonorms}, respectively, then
\[ 1 + \frac{r^2}{\max_j r_j} \left(\sum_{i} |X_i (\lambda \Phi + 2 \pi \xi \cdot x)|_y|^2 \right)^{1/2} \gtrsim |\xi|^{1/2} \gtrsim (|\lambda| + |\xi|)^{1/2},\] and so \eqref{rapiddecay2} follows from \eqref{sizeest}, where we also use $r^{d-1} \leq |\lambda|^{-(d-1)/2}$.

Next, consider \eqref{rapiddecay1}. For any $\xi \in \Xi_1$, let $j_1$ be an index at which the minimum of $r_j$ is attained (i.e., $r_{j_1} = r$) and let $j_2$ be an index maximizing $r_j$. Each $r_j$ is at most $\lambda^{-1/2}$, so choosing $c < 1$, the strict inequality $r_{j_1} < r_{j_2}$ implies that  $r_{j_1} \neq |\lambda|^{-1/2}$, which forces the identity $r_{j_1} = |\xi_{j_1}|^{-1/2}$ to hold. Consequently
$|\xi_{j_1}|^{-1/2} \leq c |\lambda|^{-1/2}$, i.e., $|\lambda| \leq c^2 |\xi_{j_1}|$. By the triangle inequality,
\[ | \lambda \partial_{j_1} \Phi(y) + 2 \pi \xi_{j_1}| \geq |2 \pi \xi_{j_1} | - |\lambda| C_\Phi,  \]
where $C_\Phi$ is the constant in \eqref{finitephinorms}. Therefore, any positive choice of $c$ satisfying $c < \min\{ (\pi / (C_\Phi))^{1/2}, 1\}$ will give the inequality
\begin{equation} | \lambda \partial_{j_1} \Phi(y) + 2 \pi \xi_{j_1}| \geq | 2 \pi \xi_{j_1}| - |\lambda| C_\Phi \geq |\pi \xi_{j_1}|  = \pi r_{j_1}^{-2}. \label{largederiv} \end{equation}
  On the other hand, the triangle inequality also guarantees that
\begin{equation} | \lambda \partial_{j_2} \Phi(y) + 2 \pi \xi_{j_2}| \leq C_\Phi |\lambda| + 2 \pi |\xi_{j_2}| \leq \left(C_\Phi + 2 \pi \right) r_{j_2}^{-2}. \label{smallderiv} \end{equation}
Combining the inequalities \eqref{largederiv} and \eqref{smallderiv} with the uniform control on the derivatives $\partial_i \rho$ provided by \eqref{implicit} gives that the Euclidean-normalized vector field
\[ X_{j_1 j_2} := \frac{\partial_{j_2} \rho(x) }{\sqrt{  (\partial_{j_1} \rho(x))^{2} + (\partial_{j_2} \rho(x))^{2}}} \partial_{j_1}  - \frac{\partial_{j_1} \rho(x) }{\sqrt{  (\partial_{j_1} \rho(x))^{2} + (\partial_{j_2} \rho(x))^{2}}} \partial_{j_2} \]
applied to $\lambda \Phi + 2 \pi \xi \cdot x$ will satisfy
\begin{align*}
 \left|  X_{j_1 j_2} ( \lambda \Phi + 2 \pi \xi \cdot x)   \right| & \geq \left| \frac{\partial_{j_2} \rho(x) }{\sqrt{  (\partial_{j_1} \rho(x))^{2} + (\partial_{j_2} \rho(x))^{2}}} \right| \pi r_{j_1}^{-2} \\ & \qquad  - \left| \frac{\partial_{j_1} \rho(x) }{\sqrt{  (\partial_{j_1} \rho(x))^{2} + (\partial_{j_2} \rho(x))^{2}}}  \right| \left(C_\Phi + 2 \pi \right) r_{j_2}^{-2}. \end{align*}
 Since \eqref{finiterhonorms} and \eqref{implicit} provide uniform bounds above and below on $|\partial_{j_1} \rho(x) / \partial_{j_2} \rho(x)|$, the inequality $r_{j_2}^{-2} \geq c^{2} r_{j_1}^{-2}$ guarantees that 
\begin{equation}  \left|  X_{j_1 j_2}  (\lambda \Phi + 2 \pi \xi \cdot x  ) \right| \gtrsim r_{j_1}^{-2} \label{nstat} \end{equation}
for all $x \in B_1$ provided that $c$ is sufficiently small depending on admissible constants. Just as was encountered in Proposition \ref{firstsize}, one has
\[ \sum_{\ell=2}^N \left| (r X_{j_1 j_2})^{\ell} (\lambda \Phi + 2 \pi \xi \cdot x) \right|^{\frac{N-\ell}{k}} \lesssim 1 \]
because $r^2 (|\lambda| + |\xi|) \lesssim 1$.  Supposing that $\eta$ is any smooth function on $B_1$ supported on a product of intervals $\tilde I := \tilde I_1 \times \cdots \times \tilde I_{2d}$, each of length at most $r$, and satisfying 
\begin{equation} |\partial^{\alpha} \eta(x)| \leq C_{\alpha} r^{-|\alpha|}, \label{partunity} \end{equation}
for $|\alpha| \leq N$, then just as in \eqref{nlinderivest},
\[ \left| ( r X_{j_1 j_2}^*)^k ( a_{y,\xi} \eta) \right| \lesssim \prod_{j=1}^{2d} r_j^{-\frac{1}{2}} \]
for $k \in \{0,\ldots,N\}$. By \eqref{fullibp}, then
\begin{align*} &   \left| \int e^{i (\lambda \Phi(x) + 2 \pi \xi \cdot x)} a_{y,\xi}(x) \eta(x) d \sigma(x) \right| \\ & \qquad \lesssim \frac{r^{2d-1}}{|K|^N} \left( \prod_{j=1}^{2d} r_j^{-\frac{1}{2}} \right) \left[ 1 + \sum_{k=0}^N || r X_{j_1 j_2} (\lambda \Phi + 2 \pi \xi \cdot x) - K ||_{L^\infty(M \cap \tilde I)}^{N-k} \right]
\end{align*}
for any choice of $K$. Let $K$ equal $r X_{j_1 j_2} (\lambda \Phi + 2 \pi \xi \cdot x)$ evaluated at any point in $\tilde I$. The inequality \eqref{nstat} implies $|K| \gtrsim r^{-1}$ and the Mean Value Theorem implies
\[ || r X_{j_1 j_2} (\lambda \Phi + 2 \pi \xi \cdot x) - K ||_{L^\infty(M \cap \tilde I)} \lesssim r^2 (|\lambda| + |\xi|) \lesssim 1. \]
Therefore
\[  \left| \int e^{i (\lambda \Phi(x) + 2 \pi \xi \cdot x)} a_{y,\xi}(x) \eta(x) d \sigma(x) \right| \leq r^{N + 2d - 1} \prod_{j=1}^{2d} r_j^{-\frac{1}{2}}. \]
It is always possible to construct a partition of unity on $I_{y,\xi}$ adapted to boxes of side length $r$ with smooth cutoff functions $\eta$ satisfying \eqref{partunity} uniformly. This partition of unity can be taken to have a number of elements which is at most some fixed constant times $r^{-2d} \prod_{j} r_j$. Summing over the partition gives 
$\left| \mathcal I(y,\xi) \right| \lesssim r^{N-1} \prod_{j} r_j^{1/2}$
for all $(y,\xi) \in B_1 \times \Xi_1$, which gives \eqref{rapiddecay1} because $r_j \leq |\lambda|^{-1/2}$ for each $j$.

Finally, to prove \eqref{schwartzpart}, fix $N := 2d + 2$ (this is why we need \eqref{finiterhonorms} to hold for all $|\alpha| \leq 2d + 3$ and \eqref{finitephinorms} and \eqref{finiteanorms} to hold for all $|\alpha| \leq 2d + 2$). For all $(y,\xi) \in E$, \eqref{rapiddecay1} and \eqref{rapiddecay2} imply that $|{\mathcal I}(y,\xi)| \lesssim (|\lambda| + |\xi|)^{-(2d+1)/2} |\lambda|^{-(d-1)/2}$ because $|\lambda|$ is necessarily bounded below.
By Cauchy-Schwarz,
\begin{align*}
\int_{B_1 \times \R^{2d}} &  \chi_{E}(y,\xi) |\mathcal I(y,\xi)| \prod_{j=1}^{2d} |f_j(y_j,\xi_j)| dy d \xi \\
& \lesssim |\lambda|^{-\frac{d-1}{2}} \left[ \int_{B_1 \times \R^{2d}} (|\lambda| + |\xi|)^{-(2d+1)} dy d \xi \right]^{\frac{1}{2}} \prod_{j=1}^{2d} ||f_j||_{L^2(\R \times \R)} \\
& \lesssim |B_1|^{\frac{1}{2}} |\lambda|^{-\frac{d-1}{2}} \left[ \int_{\R^{2d}} (|\lambda| + |\xi|)^{-(2d+1)} d \xi \right]^{\frac{1}{2}}\prod_{j=1}^{2d} ||f_j||_{L^2(\R \times \R)}.
\end{align*}
Making the change of variables $\xi \mapsto |\lambda| \xi$ gives
\begin{align*}
\int_{B_1 \times \R^{2d}} &  \chi_{E}(y,\xi) |\mathcal I(y,\xi)| \prod_{j=1}^{2d} |f_j(y_j,\xi_j)| dy d \xi \\
& \lesssim |\lambda|^{-\frac{d-1}{2}} \left[ |\lambda|^{-1} \int_{\R^{2d}} (1 + |\xi|)^{-(2d+1)} d \xi \right]^{\frac{1}{2}}\prod_{j=1}^{2d} ||f_j||_{L^2(\R \times \R)}
\end{align*}
which gives \eqref{schwartzpart} by virtue of the triangle inequality.
\end{proof}

\begin{proposition}
Let $\Xi_2$ be the complement of $\Xi_1 \subset \R^{2d}$. For each $(y,\xi)$, let $\tau_0(\lambda,y,\xi) := - \nabla \rho(y) \cdot (\lambda \Phi(y) + 2 \pi \xi) / |\nabla \rho(y)|^2$. For any positive integer $N \leq 2d+3$, any sufficiently small $c' > 0$ depending only on admissible constants, and any $(y,\xi) \in B_1 \times \Xi_2$, \label{integprop}
\begin{equation} |\mathcal I (y ,\xi) | \lesssim \chi_{|\rho(y)| \leq c r} r^d \int_{|\tau - \tau_0 | \leq c' r^{-1}} \frac{d \tau}{\left( 1 + r |\lambda \nabla \Phi(y) + 2 \pi \xi  + \tau \nabla \rho(y)) | \right)^{N}}  \label{sizeest2} \end{equation}
for some positive $c$ depending only on admissible constants. 
\end{proposition}
\begin{proof}
For all $\xi \in \Xi_2$, all $r_j$ are comparable and so $r \approx \min\{ \lambda^{-1/2}, |\xi|^{-1/2}\}$.
By definition of $\tau_0$ and \eqref{vcalc},
\[ 1 + r \left( \sum_{i=1}^{2d}  |X_i ( \lambda \Phi + 2 \pi \xi \cdot x)|_{y}|^2 \right)^{\frac{1}{2}} =   1 + r | \lambda \Phi(y) + 2 \pi \xi + \tau_0 \nabla \rho(y)|. \]
Now for any $\tau$ in the interval $|\tau - \tau_0| \leq c'r^{-1}$, the triangle inequality gives that
\[ 1 + r | \lambda \Phi(y) + 2 \pi \xi + \tau_0 \nabla \rho(y)| \approx 1 + r | \lambda \Phi(y) + 2 \pi \xi + \tau \nabla \rho(y)|.\]
The inequality \eqref{sizeest2} follows from \eqref{sizeest} because $r_{j_0}^{-1} \prod_{j=1}^{2d} r_j^{1/2} \approx r^{d-1}$ and because the integrand on the right-hand side of \eqref{sizeest2} is greater than or comparable to 
\begin{equation} \left( 1 + r \left( \sum_{i=1}^{2d}  |X_i ( \lambda \Phi + 2 \pi \xi \cdot x)|_{y}|^2 \right)^{\frac{1}{2}} \right)^{-N} \label{integrandagain} \end{equation}
for every $\tau$ in the interval $|\tau - \tau_0| \leq c' r^{-1}$. (In particular, \eqref{sizeest2} has a factor of $r^d$ instead of $r^{d-1}$ precisely because \eqref{integrandagain} is bounded above by $r$ times the integral over $\tau$ appearing in \eqref{sizeest2}.)
The fact that $I(y,\xi) = 0$ outside the region $|\rho(y)| \leq c r$ was observed in Proposition \ref{firstsize}. 
\end{proof}

\subsection{Main contribution}
\label{mainsec}

By \eqref{schwartzpart}, the proof of \eqref{mainresult} is reduced to establishing an inequality of the form
\begin{equation} \int_{E}  \frac{r^d \chi_{|\rho(y)| \lesssim r} \prod_{j=1}^{2d} |f_j(y_j,\xi_j)|d y d \tau d \xi }{ ( 1 + r | \lambda \nabla \Phi(y) + 2 \pi \xi + \tau \nabla \rho(y)| )^{2d+1}}   \lesssim |\lambda|^{-\frac{d-1}{2}} \prod_{j=1}^{2d} ||f_j||_{L^2(\R \times \R)} \label{smallregion} \end{equation}
with $E := \set{(y,\tau,\xi)}{ |\tau - \tau_0| \leq c' r^{-1} \mbox{ and either } |\xi| \lesssim |\lambda| \mbox{ or } |\tau_0| \gtrsim r^{-2}}$. Here the integrand in \eqref{smallregion} is comparable to the upper bound of $|\mathcal I(y,\xi)|$ provided by \eqref{sizeest} because for any pair $(y,\xi)$ such that neither \eqref{rapiddecay1} nor \eqref{rapiddecay2} hold, necessarily $\max_{j} r_j \lesssim \min_j r_j$. There are also two key inequalities which hold uniformly for all such pairs $(y,\xi)$: the first is that $r_j^{-2} \approx |\lambda| + |\xi|$ for each $j$ since $r^{-2} = \max \{ |\lambda|, \max_j |\xi_j| \}  \approx |\lambda| + |\xi|$, and the second is that any $\tau$ in the interval $|\tau - \tau_0| \leq c' r^{-1}$ will satisfy $|\lambda| + |\tau| \approx |\lambda| + |\xi|$. To prove this latter inequality, first fix the sufficiently small constant $c'$ from Proposition \ref{integprop} to equal the constant $c'$ in \eqref{smalltao}. With this identification of constants, 
if \eqref{smalltao} fails, $|\tau| \geq |\tau_0| - |\tau-\tau_0| \geq c' r^{-2} - c' r^{-1} \gtrsim r^{-2}$, which means that $|\lambda| + |\tau| \gtrsim r^{-2} \approx |\lambda| + |\xi|$.  One must then have $|\lambda| + |\tau| \approx |\lambda| + |\xi|$ because $|\tau_0| \lesssim |\lambda| + |\xi|$.  If \eqref{smalltao} does not fail, then $|\xi| \leq c'' |\lambda|$  and thus $|\tau| \lesssim |\lambda|$, which implies $|\tau| + |\lambda| \approx |\lambda| \approx r^{-2}$. 
It suffices, then, to prove \eqref{smallregion} with the set $E$ replaced by those triples $(y,\tau,\xi) \in B_1 \times \R \times \R^{2d}$ satisfying
\begin{equation}
r_1^{-2}  \approx  \cdots \approx r_{2d}^{-2} \approx |\lambda| + |\xi|  \approx |\lambda| + |\tau|.
\end{equation}

The proof of this modified version of \eqref{smallregion} is an elementary interpolation argument. The idea is to partition the indices $\{1,\ldots,2d\}$ into two sets $\{i_1,\ldots,i_{d}\}$ and $\{j_1,\ldots,j_d\}$ which give a nonzero determinant {\it a la} \eqref{mainhyp}. The partitioning must be done locally in $y$ to account for the fact that no one choice of partition is guaranteed to work at all points. With this partition and localization in place, one assumes that $f_{i_1},\ldots,f_{i_d}$ belong to $L^\infty$ and that $f_{j_1},\ldots,f_{j_d}$ belong to $L^1$. For these spaces, the inequality analogous to \eqref{smallregion} is essentially just a consequence of the change of variables formula. The key information needed to apply change of variables successfully is the local invertibility of the map and an estimate of its Jacobian determinant. Both of these are provided by the following lemma:
\begin{lemma}
Suppose $\rho(u,v)$ and $\Phi(u,v)$ are $C^3$ functions on some open set $U \times V \subset \R^d \times \R^d$, where $U$ and $V$ are both $d$-fold products of intervals and let $\omega := (\omega_1,\omega_2)$ denote points on the unit circle in $\R^2$. Suppose \label{cov}
\begin{equation} 
\left| \det \left[ \begin{array}{cccc} 0 & \partial_{u_1} \rho & \cdots & \partial_{u_d} \rho \\
\partial_{v_1} \rho  & \partial^2_{v_1 u_1} \left( \omega_1 \Phi + \omega_2 \rho \right) & \cdots &  \partial^2_{v_1 u_d} \left( \omega_1 \Phi + \omega_2 \rho \right) \\
 \vdots & \vdots & \ddots & \vdots \\
\partial_{v_d} \rho  & \partial^2_{v_d u_1} \left( \omega_1 \Phi + \omega_2 \rho \right) & \cdots &  \partial^2_{v_d u_d} \left( \omega_1 \Phi + \omega_2 \rho \right)
\end{array} \right]  \right| \geq c \label{themaindet} \end{equation}
at some point $p := (u_p,v_p,\omega_p) \in U \times V \times {\mathbb S}^1$.
 Then there exist metric balls $U_0$, $V_0$, and $W_0$, centered at $u_p$, $v_p$, and $\omega_p$, respectively, 
such that for all $v \in V_0$ and all $\lambda \in \R$, the map
\[ \Psi_{v,\lambda} (u, \tau) :=  (\rho(u,v) , \lambda \nabla_v \Phi(u,v) + \tau \nabla_v \rho(u,v)) \in \R \times \R^d \]
is $1$--$1$ on the open set
\[ {\mathcal U}_\lambda := \set{ (u,\tau) \in U_0 \times \R}{ \frac{(\lambda,\tau)}{\sqrt{\lambda^2 + \tau^2}} \in W_0}. \]
The radii of the balls $U_0,V_0,$ and $W_0$ can be taken to depend only on $c$, $d$ and the $C^3$-norms of $\rho$ and $\Phi$ on $U \times V$.
Moreover, for $(u,\tau) \in {\mathcal U}_\lambda$,  the Jacobian determinant of $\Psi_{v,\lambda}$ with respect to $(u,\tau)$ satisfies
\begin{equation} \left| \det \frac{\partial \Psi_{v,\lambda}}{\partial (u,\tau)} \right| \approx (\lambda^2 + \tau^2)^{\frac{d-1}{2}} \label{jacobian} \end{equation}
for implicit constants depending only on $c$, $d$, and the $C^3$-norms of $\rho$ and $\Phi$.
\end{lemma}
\begin{proof}
Consider the variant of the determinant in \eqref{themaindet} in which every entry of the matrix is evaluated at its own triple $(u',v',\omega')$ near the triple $(u_p,v_p,\omega_p)$ (so that the modified determinant is a function of many triples $(u',v',\omega')$ rather than merely a single one).
By the Mean Value Theorem, there must be some radius $\delta$ depending only on $c$ and the derivatives of $\Phi$ and $\rho$ through order $3$ such that this generalized determinant must have magnitude greater than $c/2$ whenever every triple $(u',v',\omega')$ appearing in the determinant satisfies $|u' - u_p| < \delta$, $|v' - v_p| < \delta|$ and $|\omega' - \omega_p| < \delta$. Let $U_0, V_0,$ and $W_0$ be exactly these $\delta$-neighborhoods of $u_p,v_p,$ and $\omega_p$, respectively.

Given any $u,u' \in U_0$, let
\[ \angg{f} := \int_0^1 f( \theta(u' - u) + u,v) d \theta \]
for any continuous function $f$ on $U_0 \times V_0$. 
Consider $v \in V_0$ and $\lambda \in \R$ to be fixed and
suppose that the pairs $(u,\tau), (u',\tau') \in \mathcal U_\lambda$ satisfy $\Psi_{v,\lambda}(u,\tau) = \Psi_{v,\lambda}(u',\tau')$.
By the Fundamental Theorem of Calculus, it follows that
\begin{align*}
 0  = & \rho(u',v) - \rho(u,v) = \sum_{j=1}^d  (u'_j-u_j) \angg{\partial_{u_j} \rho}, \\
0  = & \lambda \left[ \partial_{v_i} \Phi(u',v) - \partial_{v_i} \Phi(u,v) \right] + \tau' [ \partial_{v_i} \rho(u',v) - \partial_{v_i} \rho(u,v)] + (\tau' - \tau) \partial_{v_i} \rho(u,v)  \\
 = & \sum_{j=1}^d  (u'_j-u_j) \angg{\partial^2_{v_i u_j} (\lambda \Phi + \tau' \rho)} + (\tau' - \tau) \partial_{v_i} \rho(u,v), 
\end{align*}
for each $i \in \{1,\ldots,d\}$. This implies that the matrix product
\[ \left[ \begin{array}{cccc} 0 & \angg{ \partial_{u_1} \rho} & \cdots & \angg{\partial_{u_d} \rho} \\
\partial_{v_1} \rho(u,v) &\angg{ \partial^2_{v_1u_1} ( \lambda \Phi + \tau' \rho)} & \cdots & \angg{ \partial^2_{v_1u_d} ( \lambda \Phi + \tau' \rho)} \\
\vdots & \vdots & \ddots & \vdots \\
\partial_{v_d} \rho(u,v) & \angg{ \partial^2_{v_d u_1} ( \lambda \Phi + \tau' \rho)}  & \cdots & \angg{ \partial^2_{v_d u_d} ( \lambda \Phi + \tau' \rho)}
\end{array} \right] \left[ \begin{array}{c} \tau' - \tau \\ u_1' - u_1 \\ \vdots \\ u_d' - u_d \end{array} \right] \]
must be zero.
If the determinant of the matrix is nonzero, the matrix must be invertible, which means that $(u,\tau) = (u',\tau')$.
The the determinant of this matrix is homogeneous of degree $d-1$ in the variables $(\lambda,\tau')$ since every nonzero product in the permutation expansion of the determinant will have exactly one entry from the first column and exactly one entry from the first row (because the only term belonging to both the first row and first column is zero), and because all other entries are homogeneous of degree $1$ in $(\lambda,\tau')$. Thus to show that $(u,\tau) = (u',\tau')$ it suffices to show that
\[ \det \left[ \begin{array}{cccc} 0 & \angg{ \partial_{u_1} \rho} & \cdots & \angg{\partial_{u_d} \rho} \\
\partial_{v_1} \rho(u,v) &\angg{ \partial^2_{v_1u_1} ( \omega_1' \Phi + \omega_2' \rho)} & \cdots & \angg{ \partial^2_{v_1u_d} ( \omega_1' \Phi + \omega_2' \rho)} \\
\vdots & \vdots & \ddots & \vdots \\
\partial_{v_d} \rho(u,v) & \angg{ \partial^2_{v_d u_1} ( \omega_1' \Phi + \omega_2' \rho)}  & \cdots & \angg{ \partial^2_{v_d u_d} ( \omega_1' \Phi + \omega_2' \rho)}
\end{array} \right] \neq 0 \]
where $(\omega_1',\omega_2') := (\lambda,\tau')/\sqrt{\lambda^2 + (\tau')^2} \in {\mathbb S}^1$.
By multilinearity of the determinant in its columns and by expanding the averages $\angg{\cdot}$ in each column to express the column as an average over the column along the line segment from $u$ to $u'$, this determinant can be written as an average of determinants of the form initially considered, specifically determinants of the same form that appears in \eqref{themaindet} with the difference that each column is a function of its own distinct $u$-variable. By assumption on $U_0$, these determinants are never zero, and so the average determinant will also never be zero.

To estimate the Jacobian determinant \eqref{jacobian}, the key is to note that the Jacobian matrix $\partial \Psi_{v,\lambda} / \partial (u,\tau)$ will (up to a permutation of rows and columns) have the same structure as the main determinant \eqref{themaindet} with $\omega_1$ replaced by $\lambda$ and $\omega_2$ replaced by $\tau$. As already observed, such a determinant will be homogeneous of degree $d-1$ in $(\lambda,\tau)$ and because the main determinant \eqref{themaindet} is bounded uniformly above and below on $U_0 \times V_0 \times W_0$ with constants depending only on $c$, $d$, and the $C^3$ norms of $\rho$ and $\Phi$, the desired relationship \eqref{jacobian} must hold.
\end{proof}

By Lemma \ref{cov} and the hypothesis \eqref{mainhyp}, for any value of $\lambda$,  $B_1 \times \R$ is covered by boundedly many subsets $F$, the number of which depends only on admissible constants, such that for each $F$, there is some partition of the indices $\{1,\ldots,2d\}$ into disjoint sets $\{i_1,\ldots,i_d\}$ and $\{j_1,\ldots,j_d\}$ for which
\begin{itemize}
\item For any fixed values of $(x_{i_1},\ldots,x_{i_d})$, the map
\[ (x_{j_1},\ldots,x_{j_d} , \tau) \mapsto ( \rho(x), \lambda \partial_{i_1} \Phi(x) + \tau \partial_{i_1} \rho(x) ,\ldots, \lambda \partial_{i_d} \Phi(x) + \tau \partial_{i_d} \rho(x))  \]
is one-to-one on the set $\set{(x_{j_1},\ldots,x_{j_d})}{ (x_1,\ldots,x_{2d} , \tau) \in F}$ and has Jacobian determinant with magnitude comparable to $(|\lambda| + |\tau|)^{d-1}$ up to multiplicative factors depending only on admissible constants.
\item For any fixed values of $(x_{j_1},\ldots,x_{j_d})$, the map
\[ (x_{i_1},\ldots,x_{i_d} , \tau) \mapsto ( \rho(x), \lambda \partial_{j_1} \Phi(x) + \tau \partial_{j_1} \rho(x) ,\ldots, \lambda \partial_{j_d} \Phi(x) + \tau \partial_{j_d} \rho(x))  \]
is one-to-one on the set $\set{(x_{i_1},\ldots,x_{i_d})}{ (x_1,\ldots,x_{2d} , \tau) \in F}$ and has Jacobian determinant with magnitude comparable to $(|\lambda| + |\tau|)^{d-1}$ up to multiplicative factors depending only on admissible constants.
\end{itemize}

Returning to the proof of \eqref{smallregion}, let $F$ be any such region in $B_1 \times \R$ as described above, and  let $E_F := \set{ (y,\tau,\xi) \in E}{ (y,\tau) \in F}$.
\begin{align*} & \int_{E_F} \frac{r^d \chi_{|\rho(y)| \lesssim r}}{ ( 1 + r | \lambda \nabla \Phi(y) + 2 \pi \xi + \tau \nabla \rho(y)| )^{2d+1}} \prod_{j=1}^{2d} |f_j(y_j,\xi_j)| d y d \tau d \xi \\
& \leq  \mathop{\mathop{\mathrm{ess.sup}}_{y_{j_1},\xi_{j_1}, \ldots}}_{\ldots, y_{j_d},\xi_{j_d}} \int \frac{r^d  \chi_E(y,\tau,\xi) \chi_{|\rho(y)| \lesssim r} d y_{i_1} d \xi_{i_1} \cdots d y_{i_d} d \xi_{i_d} d \tau}{ ( 1 + r | \lambda \nabla \Phi(y) + 2 \pi \xi + \tau \nabla \rho(y)| )^{2d+1}}   \prod_{k=1}^d ||f_{i_k}||_\infty ||f_{j_k}||_1.
\end{align*}

Several observations are in order. First, on $E$ each $r_j$ is comparable to $r$, and so up to a factor which depends only on admissible constants, one may simply replace $r$ by $r_{j_1}$. This is advantageous because $r_{j_1}$ is independent of the variables of integration $y_{i_k}$, $\xi_{i_k}$ for $k \in \{1,\ldots,d\}$. Next, if $P_{j_1 \cdots j_d}$ denotes orthogonal projection from $\R^{2d}$ onto the coordinates indexed by $j_1,\ldots,j_d$, then
\[ \int_{\R^d} \frac{d \xi_{i_1} \cdots d \xi_{i_d}} {(1 + r_{j_1} | v + 2 \pi \xi|)^{2d+1}} \lesssim \frac{r_{j_1}^{-d}}{(1 + r_{j_1} |P_{j_1 \cdots j_d} (v + 2 \pi \xi)|)^{d+1}} \]
for any $v \in \R^{2d}$, which follows by change of variables and the observation that $1 + r_{j_1} |v + 2 \pi \xi| \approx 1 + r_{j_1} |P_{j_1 \cdots j_d} ( v + 2 \pi \xi)| + r_{j_1} |P_{i_1 \cdots i_d} (v + 2 \pi \xi)|$.
It follows that
\begin{align} \mathop{\mathop{\mathrm{ess.sup}}_{y_{j_1},\xi_{j_1}, \ldots}}_{\ldots, y_{j_d},\xi_{j_d}} & \int \frac{r^d  \chi_{E_F}(y,\tau,\xi) \chi_{|\rho(y)| \lesssim r} d y_{i_1} d \xi_{i_1} \cdots d y_{i_d} d \xi_{i_d} d \tau}{ ( 1 + r | \lambda \nabla \Phi(y) + 2 \pi \xi + \tau \nabla \rho(y)| )^{2d+1}} \nonumber \\
& \lesssim  \mathop{\mathop{\mathrm{ess.sup}}_{y_{j_1},\xi_{j_1}, \ldots}}_{\ldots, y_{j_d},\xi_{j_d}} \int \frac{\chi_{\tilde F}(y,\tau) \chi_{|\rho(y)| \lesssim r_{j_1}} d y_{i_1}  \cdots d y_{i_d}  d \tau}{ ( 1 + r_{j_1} | P_{j_1 \cdots j_d} (\lambda \nabla \Phi(y) + 2 \pi \xi + \tau \nabla \rho(y) ) | )^{d+1}},  \label{lastrhs}
\end{align}
where $\tilde F \subset F$  is the subset of points which satisfy $|\tau| + |\lambda| \approx r_{j_1}^{-2}$ for the given value of $r_{j_1}$.
The integral \eqref{lastrhs} can be estimated by writing $(y_{i_1},\ldots,y_{i_d},\tau)$ as a function of $\rho(y)$ and $P_{j_1 \cdots j_{d}} (\lambda \nabla \Phi(y) + 2 \pi \xi + \tau \nabla \rho(y) )$ for fixed $y_{j_1},\xi_{j_1},\ldots, y_{j_d},$ and $\xi_{j_d}$. By \eqref{jacobian}, the Jacobian determinant of the map $(y_{i_1},\ldots,y_{i_d},\tau) \mapsto (\rho(y), P_{j_1 \cdots j_{d}} (\lambda \nabla \Phi(y) + 2 \pi \xi + \tau \nabla \rho(y) ))$ is comparable to $(|\lambda| + |\tau|)^{d-1} \approx r_{j_1}^{-2(d-1)}$. Therefore by the change of variables formula and \eqref{lastrhs},
\begin{align*} \mathop{\mathop{\mathrm{ess.sup}}_{y_{j_1},\xi_{j_1}, \ldots}}_{\ldots, y_{j_d},\xi_{j_d}} & \int \frac{r^d  \chi_{E_F}(y,\tau,\xi) \chi_{|\rho(y)| \lesssim r} d y_{i_1} d \xi_{i_1} \cdots d y_{i_d} d \xi_{i_d} d \tau}{ ( 1 + r | \lambda \nabla \Phi(y) + 2 \pi \xi + \tau \nabla \rho(y)| )^{2d+1}} \nonumber \\
& \lesssim  \mathop{\mathop{\mathrm{ess.sup}}_{y_{j_1},\xi_{j_1}, \ldots}}_{\ldots, y_{j_d},\xi_{j_d}} r_{j_1}^{2d - 2} \int \frac{\chi_{|s| \lesssim r_{j_1}} du_1 \cdots d u_d ds}{(1 + r_{j_1} |u|)^{d+1}}  \\
& \lesssim  \mathop{\mathop{\mathrm{ess.sup}}_{y_{j_1},\xi_{j_1}, \ldots}}_{\ldots, y_{j_d},\xi_{j_d}} r_{j_1}^{d - 2} \int \frac{\chi_{|s| \lesssim r_{j_1}} du_1 \cdots d u_d ds}{(1 +  |u|)^{d+1}} \lesssim \mathop{\mathop{\mathrm{ess.sup}}_{y_{j_1},\xi_{j_1}, \ldots}}_{\ldots, y_{j_d},\xi_{j_d}} r_{j_1}^{d - 1} \lesssim |\lambda|^{-\frac{d-1}{2}}.
\end{align*}
Thus
\begin{equation*} \int_{E_F}  \frac{r^d \chi_{|\rho(y)| \lesssim r} \prod_{j=1}^{2d} |f_j(y_j,\xi_j)|d y d \tau d \xi }{ ( 1 + r | \lambda \nabla \Phi(y) + 2 \pi \xi + \tau \nabla \rho(y)| )^{2d+1}}   \lesssim |\lambda|^{-\frac{d-1}{2}} \prod_{k=1}^{d} ||f_{i_k}||_{\infty} ||f_{j_k}||_1. \end{equation*}
A completely symmetric argument establishes the same inequality with the roles of $i_1,\ldots,i_d$ and $j_1,\ldots,j_d$ reversed, and a subsequent interpolation gives the analogue of \eqref{smallregion} for integration over $E_F$. The proof of \eqref{mainresult} then follows by summing over the regions $F$ on which Lemma \ref{cov} applies to our change of variables map for appropriate choice of indices $\{i_1,\ldots,i_d\}$ and $\{j_1,\ldots,j_d\}$.

\bibliography{mybib}

\end{document}